\documentclass[a4paper,11pt,leqno]{amsart}

\usepackage[matrix,arrow,tips,curve]{xy}
\usepackage{pictexwd, dcpic} 
\usepackage[english]{babel}
\usepackage{amsmath}
\usepackage{amssymb}    
\usepackage{graphicx}
\usepackage{mathrsfs}
\usepackage{enumerate} 
\usepackage{psfrag}

\usepackage{natbib}

 \psfrag{E}{\fontsize{20}{20}$E$}
 \psfrag{F}{\fontsize{20}{20}$F$}
 \psfrag{B}{\fontsize{20}{20}$\Delta_0$}
 \psfrag{l}{\fontsize{20}{20}$\ell$}
 \psfrag{E5}{\fontsize{20}{20}$\wi{E}$}
 \psfrag{E0}{\fontsize{20}{20}$E_0$}
 \psfrag{E1}{\fontsize{20}{20}$E_1$}
 \psfrag{E2}{\fontsize{20}{20}$E_2$}
 \psfrag{A1}{\fontsize{20}{20}$A_1$}
 \psfrag{A2}{\fontsize{20}{20}$A_2$}
 \psfrag{G1}{\fontsize{20}{20}$E_0'$}
 \psfrag{G}{\fontsize{20}{20}$E_0'$}
 \psfrag{G2}{\fontsize{20}{20}$E_0''$}
 \psfrag{D1}{\fontsize{20}{20}$D_1$}
 \psfrag{D}{\fontsize{20}{20}$D$}
 \psfrag{D2}{\fontsize{20}{20}$D_2$}
 \psfrag{R1}{\fontsize{20}{20}$R_1$}
 \psfrag{R2}{\fontsize{20}{20}$\w{R}_2$}
 \psfrag{R3}{\fontsize{20}{20}$\w{R}_3$}
 \psfrag{X3}{\fontsize{20}{20}$X_3$}
 \psfrag{R4}{\fontsize{20}{20}$\w{R}_4$}
 \psfrag{E3}{\fontsize{20}{20}$E_3$}
 \psfrag{E4}{\fontsize{20}{20}$E_4$}
 \psfrag{S3}{\fontsize{20}{20}$\sigma(E_3)$}
 \psfrag{S4}{\fontsize{20}{20}$\sigma(E_4)$}

\newtheorem{thm}{Theorem}[section]
\newtheorem{lemma}[thm]{Lemma}
\newtheorem{proposition}[thm]{Proposition}
\newtheorem{corollary}[thm]{Corollary}

\newtheorem{question}{Question}

\theoremstyle{definition}
\newtheorem{remark}[thm]{Remark}

\newcommand{\w}{\widetilde}

\newcommand{\wi}{\widehat}

\newcommand{\Sing}{\operatorname{Sing}}

\newcommand{\Exc}{\operatorname{Exc}}

\newcommand{\codim}{\operatorname{codim}}
\newcommand{\Eff}{\operatorname{Eff}}
\newcommand{\Mov}{\operatorname{Mov}}
\newcommand{\Nef}{\operatorname{Nef}}
\newcommand{\Pic}{\operatorname{Pic}}
\newcommand{\Bs}{\operatorname{Bs}}
\newcommand{\Sym}{\operatorname{Sym}}

\title[]{Positivity of anticanonical divisors from the viewpoint of Fano conic bundles}
\author{E. A.~Romano}
\address{Universit\`a di Torino,
Dipartimento di Matematica,
via Carlo Alberto 10,
10123 Torino - Italy}
\email{eleonoraanna.romano@unito.it}

\begin{document}
\maketitle

\begin{abstract}
We give the first examples of flat fiber type contractions of Fano manifolds onto varieties that are not weak Fano, and we prove that these morphisms are Fano conic bundles. We also review some known results about the interaction between the positivity properties of anticanonical divisors of varieties of contractions. 
\end{abstract}
\section{Introduction}
\noindent Given a contraction $\varphi\colon X\to Y$, \textit{i.e.} a surjective morphism with connected fibers between normal complex projective varieties, it is a natural and fundamental problem to try to understand the link between the positivity properties of the anticanonical divisors of our varieties. This matter can be viewed as a Positivity Problem for the anticanonical divisor of $Y$.

Assume that $X$, $Y$ are smooth. In \cite[Corollary 2.9]{K}, Koll$\acute{\text{a}}$r, Miyaoka and Mori proved that when such a morphism is smooth, and $X$ is Fano, then so is $Y$. Under the same assumption of smoothness of the morphism, in \cite[Theorem 1.1]{FUJ} Fujino and Gongyo showed that if $X$ is weak Fano (that is $-K_{X}$ is nef and big) then so is $Y$; while in \cite[Theorem 1.1]{BIR} Birkar and Chen proved that if $-K_{X}$ is semiample, then so is $-K_{Y}$.

Weaker results are known when $\varphi$ is not smooth: for instance, in \cite[Theorem 2.9]{PROK} Prokhorov and Shokurov proved that if $\varphi\colon X\to Y$ is a contraction and $X$ is log Fano, then $Y$ is also log Fano, in particular $-K_{Y}$ is big. We recall that a normal projective variety $X$ is log Fano if there exists an effective $\mathbb{Q}$-divisor $\bigtriangleup$ of $X$ such that $-(K_{X}+\bigtriangleup)$ has a positive multiple which is Cartier and ample and the pair $(X, \bigtriangleup)$ is klt (see \cite{KOLLAR} for the definition of klt). Moreover, we refer the reader to \cite[$\S$3.6]{DEB} for an overview and proofs of some results mentioned above.


Let us denote by $\mathcal{N}_{1}(X)$ the $\mathbb{R}$-vector space of one-cycles with real coefficients, modulo numerical equivalence, whose dimension is the \textit{Picard number} $\rho_{X}$. A \textit{Fano conic bundle} $f\colon X\to Y$ is a contraction where $X$ is smooth and Fano, and such that all fibers of $f$ are one-dimensional. We say that a Fano conic bundle $f\colon X\to Y$ is \textit{elementary} if $\rho_{X}-\rho_{Y}=1$, otherwise $f$ is called \textit{non-elementary}. We refer the reader to \cite{IO} for a detailed study of non-elementary Fano conic bundles. 

In this note we discuss some recent developments about the Positivity Problem introduced above from the viewpoint of Fano conic bundles, and 
our main goal is to construct some new examples. 

In \cite{WIS} Wi$\acute{\text{s}}$niewski posed a question related to our positivity problem: given a Fano conic bundle $f\colon X\to Y$, is $Y$ also Fano? In \cite[$\S$5, 5.2]{WIS} he gave an example of Fano conic bundle onto a variety that is not Fano but it is weak Fano. This example was generalized by Debarre in \cite[Example 3.16 (3)]{DEB}. Moreover, in \cite[footnote 14]{DEB} Debarre observed that Wi$\acute{\text{s}}$niewski's construction does not seem able to provide an example of Fano conic bundle whose target is not weak Fano. In this paper we overcome this problem, in fact with a modification of Wi$\acute{\text{s}}$niewski's construction we get the first examples of Fano conic bundles onto varieties that are not weak Fano, by proving the following result:

\begin{thm} \label{new_examples}
	For every $m\in \mathbb{Z}$, $m\geq 2$ there exists an elementary Fano conic bundle $f\colon X\to Y$ where $\dim{X}=3(m+1)$, and $-K_{Y}$ is not nef. 
\end{thm} 

Notice that there exist examples where $f\colon X\to Y$ is a flat and not smooth contraction, with $X$ weak Fano and $-K_{Y}$ not nef (see \cite[Example 4.6]{FUJ}). Theorem \ref{new_examples} implies something stronger: $-K_{Y}$ is not always nef even if $X$ is a Fano manifold. 

We conclude with Section \ref{conclusion} where some open problems are presented. 
\section{Preliminaries}
\subsection{Notation and terminology} \label{notation}
We work over the field of complex numbers. Let $X$ be a smooth projective variety with arbitrary dimension $n$. 

$\mathcal{N}_{1}(X)$ (respectively, $\mathcal{N}^{1}(X)$) is the $\mathbb{R}$-vector space of one-cycles (respectively, Cartier divisors) with real coefficients, modulo numerical equivalence. 

$\dim{\mathcal{N}_{1}(X)}=\dim{\mathcal{N}^{1}(X)}=: \rho_{X}$ is the \textit{Picard number} of $X$.

Let $C$ be a one-cycle of $X$, and $D$ a divisor of $X$. We denote by $[C]$ (respectively, $[D]$) the numerical equivalence class in $\mathcal{N}_{1}(X)$ (respectively in $\mathcal{N}^{1}(X)$). 

We denote by $\Eff{(X)}$ the \textit{Effective cone} of $X$, that is the convex cone inside $\mathcal{N}^{1}(X)$ spanned by classes of effective divisors. 

Denote by $\Bs{(|\cdot|)}$ the base locus of a linear system. Let $D\subset X$ be a Cartier divisor. The divisor $D$ is \textit{movable} if there exists $m>0$, $m\in \mathbb{Z}$ such that $\codim{\Bs{|mD|}}\geq 2$.

We denote by $\Mov{(X)}$ the \textit{Movable cone} of $X$, that is the convex cone of $\mathcal{N}^{1}(X)$ spanned by classes of movable divisors. 

A \textit{contraction} of $X$ is a surjective morphism $\varphi\colon X\to Y$ with connected fibers, where $Y$ is normal and projective.

We denote by $\text{Exc}(\varphi)$ the \textit{exceptional locus} of $\varphi$, \textit{i.e.} the locus where $\varphi$ is not an isomorphism.

\subsection{Preliminaries on Fano conic bundles} \label{preliminaries_cb}
A contraction $f\colon X\to Y$ is a \textit{conic bundle} if $-K_{X}$ is $f$-ample and every fiber of $f$ is one-dimensional. When $X$ is Fano, we call such a contraction \textit{Fano conic bundle}.

The following theorem is due to Ando (\cite[Theorem 3.1]{ANDO}) and it is a generalization in higher dimension of Mori's result in dimension $3$ (see \cite[Theorem 3.5, (3.5.1)]{MORI}). For the second part we refer the reader to \cite[Proposition 1.2 (ii)]{B} and \cite[$\S$4]{WIS}.

An important consequence of this theorem is that conic bundles can be easily characterized among fiber type contractions of smooth varieties.
\begin{thm} [\cite{ANDO}, Theorem 3.1 (ii)]\label{Ando} Let $X$ be a smooth projective variety and let $f\colon X\to Y$ be a contraction where $-K_{X}$ is $f$-ample and every fiber is one-dimensional. Then $Y$ is smooth.
	
Moreover, there exists a locally free sheaf $\mathcal{E}$ on $Y$ of rank 3 such that the projection $\pi\colon \mathbb{P}(\mathcal{E})\to Y$ contains $X$ embedded over $Y$ as a divisor whose restriction to every fiber of $\pi$ is one-dimensional and it is an element of $|\mathcal{O}_{\mathbb{P}^{2}}(2)|$, and $f=\pi_{\mid X}$. The pushforward $f_{*} \mathcal{O}_{X}(-K_{X})$ can be taken as the above $\mathcal{E}$.
\end{thm}

By the above theorem, we get the following commutative diagram:
$$ \label{immersion}
\xymatrix{
	X\ar[dr]_f\ar@{^{(}->}[rr]^i && {\mathbb{P}(\mathcal{E})}\ar[dl]^{\pi}\\
	&{Y}&&&
}$$

Moreover, $X$ is defined by the vanishing of a section  $\sigma \in H^{0}(\mathbb{P}(\mathcal{E}), \mathcal{O}_{\mathbb{P}(\mathcal{E})}(2 \xi+ \pi^{*}M)$, where $M$ is a divisor of $Y$ and using the adjunction formula it is easy to check that  $M\sim -\det{\mathcal{E}}-K_{Y}$. 

In \cite{B}, Beauville introduced the following definition of the \textit{discriminant divisor} of a conic bundle $f\colon X\to Y$:
\begin{center} 
	$\bigtriangleup_{f}:=\{y\in Y\mid f^{-1}(y) \text{ is singular}\}$.
\end{center}

Using the same notation introduced above, we recall from \cite[$\S$1.7]{SARK} that
\begin{equation} \label{discriminant}
\bigtriangleup_{f} \in |(\det{\mathcal{E}})^{\otimes 2}\otimes \mathcal{O}_{Y}(3M)|
\end{equation}

We can express the regularity conditions of the fibers of $f$ in terms of the properties of $\bigtriangleup_{f}$. Indeed, we obtain the following explicit description of the singular locus of the discriminant divisor (see for instance \cite[$\S$2]{BARTH}, \cite[Proposition 1.8, (5.c)]{SARK}):
\begin{equation} \label{singular_discriminant}
	\Sing{(\bigtriangleup_{f})}=\{y\in Y\mid f^{-1}(y) \text{ is non-reduced}\}.
\end{equation}

Now, we recall the last results related to $\bigtriangleup_{f}$ that will be needed.

The following results are consequences of \cite[Proposition 4.3]{WIS} and its proof.
\begin{proposition} [\cite{WIS}] \label{PropWis}
	Let $f\colon X\to Y$ be a Fano conic bundle. Assume that $Y$ is not Fano. Let $C$ be a rational curve of $Y$ such that $-K_{Y}\cdot C\leq 0$. Then $C\subseteq \Sing{(\bigtriangleup_{f})}$. 
\end{proposition}

\begin{corollary} [\cite{WIS}] \label{Fano}
	Let $f\colon X\to Y$ be a Fano conic bundle. If $f$ does not have non-reduced fibers, then $Y$ is  Fano. Moreover, if $\dim{X}\leq 4$ or $\rho_{X}\leq 2$, then $Y$ is Fano.
\end{corollary}

We are going to study our Positivity Problem for the anticanonical divisors from the viewpoint of Fano conic bundles. 

We focus on the following problem studied in \cite[$\S$4]{WIS} by Wi$\acute{\text{s}}$niewski: given a Fano conic bundle $f\colon X\to Y$, is $Y$ Fano or not?

The following result arises from the study of non-elementary Fano conic bundles of \cite{IO}. In particular, it is a consequence of \cite[Theorem 1.1]{IO} and allows us to give an answer to Wi$\acute{\text{s}}$niewski's question, in many cases.
 
\begin{thm} [\cite{IO}, Corollary 1.2] \label{esempio} Let $f\colon X\to Y$ be a Fano conic bundle. If $Y$ is not Fano, then $\rho_{X}-\rho_{Y}=2$ or $\rho_{X}-\rho_{Y}=1$. If moreover $\rho_{Y}\leq 2$, or $Y$ does not have a smooth $\mathbb{P}^{1}$-fibration\footnote{A smooth $\mathbb{P}^{1}$-fibration is a smooth morphism with fibers isomorphic to $\mathbb{P}^{1}$.}, then $\rho_{X}-\rho_{Y}=1$.
\end{thm}

Thanks to the above theorem we deduce that the target of a non-elementary Fano conic bundle is often Fano. 
In the next Section we use Theorem \ref{esempio} to prove that the new examples of Fano conic bundles are elementary.

\section{New examples of elementary Fano conic bundles} \label{new}
In this section we give examples of Fano conic bundles $f\colon X\to Y$, where $-K_{Y}$ is not nef. 
Putting our construction in a more general setting, we get the first examples of equidimensional fiber type contractions of Fano varieties onto varieties $Y$ which are not weak Fano. Moreover, since we obtain equidimensional morphisms between smooth projective varieties, these are flat morphisms, and by Corollary \ref{Fano} we deduce that they have non-reduced fibers.   \\

Let $m\in \mathbb{Z}$, $m\geq 2$, and set\footnote{As usual, we follow Grothendieck's notation: for a vector bundle $\mathcal{E}$, the projectivization $\mathbb{P}(\mathcal{E})$ is the space of \textit{hyperplanes} in the fibers of $\mathcal{E}$.} $Y=\mathbb{P}_{\mathbb{P}^{3m}}(\mathcal{O}_{\mathbb{P}^{3m}}\oplus \mathcal{O}_{\mathbb{P}^{3m}}(2m)\oplus \mathcal{O}_{\mathbb{P}^{3m}}(2m))$.

One has $\rho_{Y}=2$, in particular $\mathcal{N}^{1}(Y)$ is generated by the class of a divisor $D$ associated with the line bundle $\mathcal{O}_{Y}(1)$ and by the class of the pullback $\textit{H}$ of a hyperplane in $\mathbb{P}^{3m}$. Let $\pi\colon Y\to \mathbb{P}^{3m}$ be the natural projection. Notice that $D$ is a nef divisor, not ample, and that it is globally generated (see for instance \cite[$\S$2.3, Lemma 2.3.2]{LAZ}). Since $D$ and $H$ are both nef, not ample divisors of $Y$ and $\rho_{Y}=2$, then $\Nef{(Y)}=\langle [D],[H] \rangle$ (see Figure \ref{fig:fig5}).

Denote by $V\subset Y$ the section corresponding to the trivial quotient $\mathcal{O}_{\mathbb{P}^{3m}}\oplus \mathcal{O}_{\mathbb{P}^{3m}}(2m)\oplus \mathcal{O}_{\mathbb{P}^{3m}}(2m)\twoheadrightarrow \mathcal{O}_{\mathbb{P}^{3m}}$, so that $D_{\mid V}\cong \mathcal{O}_{V}$. We have:
\begin{center}
	$-K_{Y}=3D+(1-m)H$
\end{center}
so that $-K_{Y}$ is not nef, indeed if $C$ is a line contained in $V\cong \mathbb{P}^{3m}$, then $-K_{Y}\cdot C=1-m<0$.

The section $V$ is a complete intersection of two divisors $G_{1}$, $G_{2}\sim D-2mH$. Indeed by the two possible factorizations of the trivial quotient 
\begin{center}
$\mathcal{O}_{\mathbb{P}^{3m}}\oplus \mathcal{O}_{\mathbb{P}^{3m}}(2m)\oplus \mathcal{O}_{\mathbb{P}^{3m}}(2m)\twoheadrightarrow \mathcal{O}_{\mathbb{P}^{3m}}\oplus \mathcal{O}_{\mathbb{P}^{3m}}(2m)\twoheadrightarrow \mathcal{O}_{\mathbb{P}^{3m}}$,
\end{center}
we have the immersions: 
\begin{center}
	$V\cong \mathbb{P}^{3m}=\mathbb{P}(\mathcal{O}_{\mathbb{P}^{3m}})\hookrightarrow G_{1}:=\mathbb{P}(\mathcal{O}_{\mathbb{P}^{3m}}\oplus \mathcal{O}_{\mathbb{P}^{3m}}(2m))\hookrightarrow Y$ \\
	$V\hookrightarrow G_{2}:=\mathbb{P}(\mathcal{O}_{\mathbb{P}^{3m}}\oplus \mathcal{O}_{\mathbb{P}^{3m}}(2m))\hookrightarrow Y$
\end{center}
and $V$ is given by the intersection between $G_{1}$ and $G_{2}$. For $i=1,2$ we have $\mathcal{O}_{G_{i}}(1)=D_{\mid G_{i}}$ and:
\begin{center}
	$-K_{G_{i}}= 2 D_{\mid G_{i}}+(m+1)H_{\mid G_{i}}$	
\end{center}
and by the adjunction formula we find that:
\begin{center}
	$K_{G_{i}}=(G_{i}-3D+(m-1)H)_{\mid G_{i}}$
\end{center}
so that $(G_{i}-D+2mH)_{\mid G_{i}}\sim 0$. But $\rho_{G_{i}}=\rho_{Y}=2$ and the restriction $\Pic{(Y)}\to \Pic{(G_{i})}$ is an isomorphism, then $G_{i}\sim D-2mH$.

Since $D$ is globally generated, the map associated to a sufficiently large multiple of $D$ is a contraction, which we denote by $\varphi\colon Y\to Y^{\prime}$. 

Using that $D_{\mid V}\cong \mathcal{O}_{V}$, it is easy to check that $\varphi$ contracts $V$ to a point, hence $V\subseteq \Exc{(\varphi)}$.

Notice that $\Exc{(\varphi)}=V$. Indeed, if $C$ is an irreducible curve of $Y$ such that $\varphi(C)$ is a point, since $\varphi_{*}(C)=0$ one has that $D\cdot C=0$, instead $H\cdot C>0$, then $G_{i}\cdot C<0$ for $i=1,2$, and $\Exc{(\varphi)}\subseteq G_{1}\cap G_{2}=V$. In particular, we observe that $\varphi$ is a small elementary contraction.\\ \\
\textit{Effective and Movable Cone of $Y$}. Let us denote by $g\colon Y\dashrightarrow \tilde{Y}$ the flip of $\varphi\colon Y\to Y^{\prime}$. We want to describe $\Eff{(Y)}$, by proving that:
\begin{center}
	$\Eff{(Y)}=\Mov{(Y)}=\Nef{(Y)}\cup g^{*}\Nef{(\tilde{Y})}$.	
\end{center}

To this end, we show that the divisor $D-2mH$ gives a fiber type contraction on $\tilde{Y}$, so that $[D-2mH]$ sits in the boundary of the Effective and Movable Cone of $Y$ (see Figure \ref{fig:fig5} below).

Recall that $G_{1}\sim G_{2}\sim D-2mH$. Let $i\in \{1,2\}$. First we show that in $Y$ we have $G_{i}\cdot C=0$, for every irreducible curve $C\subset G_{i}$ such that $C\cap V=\emptyset$. To this end, we take one of the two divisor $G_{i}$, for simplicity assume that it is $G_{1}$. If $C\subset G_{1}$ is an irreducible curve such that $C\cap V=\emptyset$, by recalling that $G_{1}\cap G_{2}=V$ and $G_{1}\sim G_{2}$, it follows that $G_{2}\cap C=\emptyset$, and $G_{1}\cdot C=G_{2}\cdot C=0$.

Recall that $G_{1}$ is a $\mathbb{P}^{1}$-bundle over $\mathbb{P}^{3m}$ and that $V\subset G_{1}$ is a section. Denote by ${\tilde{G}}_{1}\subset \tilde{Y}$ the transform of $G_{1}$ through the flip $g$.
 
We observe that $g_{\mid G_{1}}\colon G_{1}\to {\tilde{G}}_{1}$ is the divisorial contraction $\varphi_{\mid G_{1}}$ which sends $V$ to a point, and $\rho_{{\tilde{G}}_{1}}=1$. By what we have already shown, $\tilde{G}_{1}\cdot \Gamma=0$ for a general irreducible curve $\Gamma\subset \tilde{G}_{1}$, and since $\rho_{\tilde{G}_{1}}=1$, it follows that $\tilde{G}_{1}\cdot \tilde{\Gamma}=0$ for every irreducible curve $\tilde{\Gamma}\subset \tilde{G}_{1}$. 

Then $\tilde{G}_{1}$ is a nef divisor of $\tilde{Y}$, and being $\tilde{Y}$ a toric variety, $\tilde{G}_{1}$ is semiample and gives a contraction onto $\mathbb{P}^{1}$ which sends itself to a point. In Figure \ref{fig:fig5} we represent the cones of $Y$. 
\begin{figure}[ht]
	\centering
    \includegraphics[scale=0.35]{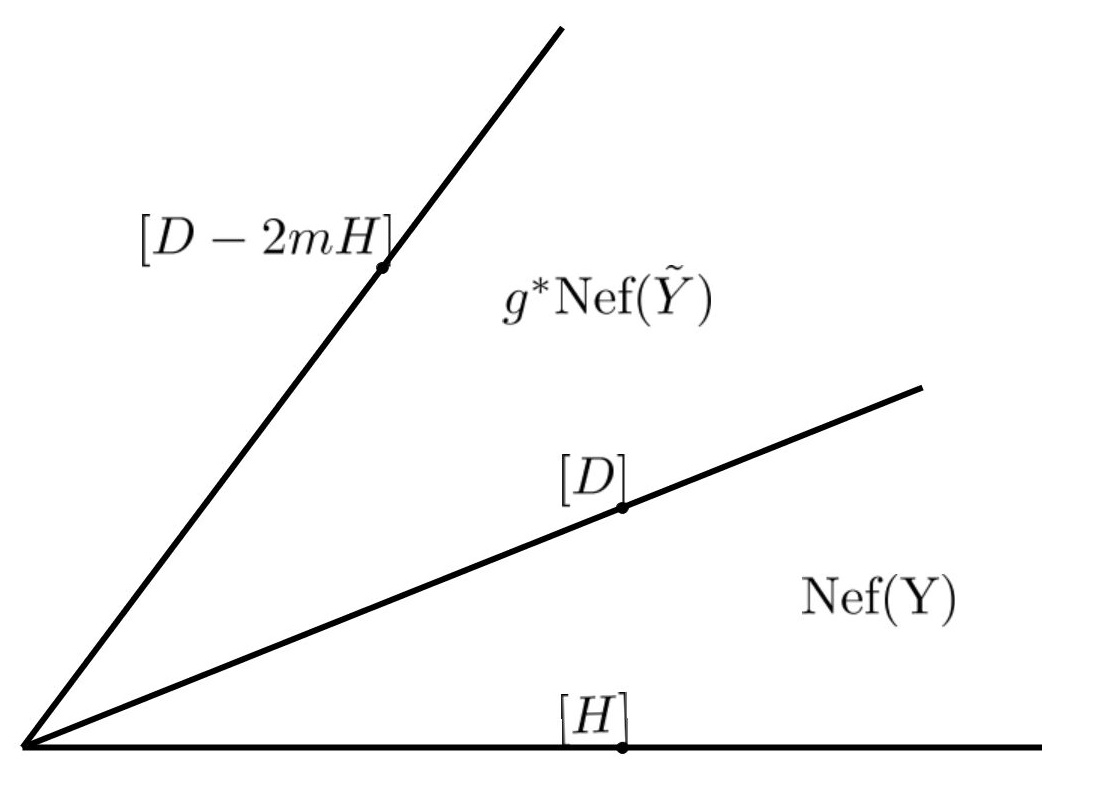}
	\caption{We denote by $g\colon Y\dashrightarrow \tilde{Y}$ the flip of the small contraction $\varphi\colon Y\to Y^{\prime}$. We get $\Eff{(Y)}=\Mov{(Y)}=\Nef{(Y)}\cup g^{*}\Nef{(\tilde{Y})}$.}
	\label{fig:fig5}
\end{figure}
\begin{lemma} \label{base_locus}
	Notation as above. Then $\Bs{(|2D-2mH|)}= \Bs{(|D-mH|)}= \Bs{(|2D-mH|)}=V$.
\end{lemma}
\begin{proof}
	For every curve $C\subset V$, we have $D\cdot C=0$ and $H\cdot C>0$, thus $(2D-2mH)\cdot C<0$, then $V\subseteq \Bs{(|2D-2mH|)}$. 
	
	We know that  $\Bs{(|2D-2mH|)}\subseteq \Bs{(|D-mH|)}$, thus we are left to prove that $\Bs{(|D-mH|)}\subseteq V$. 
	
	We have already observed that $V$ is a complete intersection of two divisors of $|D-2mH|$, hence $\Bs{(|D-2mH|)}\subseteq V$. Being $H$ globally generated, it follows that $\Bs{(|D-mH|)}\subseteq \Bs{(|D-2mH|)}$. 
	Hence we get the following inclusions: 
	\begin{center}
	$V\subseteq \Bs{(|2D-2mH|)}\subseteq \Bs{(|D-mH|)}\subseteq \Bs{(|D-2mH|)}\subseteq V$,
	\end{center} 
	so that $\Bs{(|2D-2mH|)}= \Bs{(|D-mH|)}=V$. Using the same method, we observe that $V\subseteq \Bs{(|2D-mH|)}$ and since $D$ is globally generated one has $\Bs{(|2D-mH|)}\subseteq \Bs{(|D-mH|)}=V$, thus $\Bs{(|2D-mH|)}=V$.  
\end{proof}
\begin{proposition} \label{prop_ex} Let $Y$ be as above. Let us consider the locally free sheaf $\mathcal{E}:=\mathcal{O}(D)^{\oplus 2}\oplus \mathcal{O}(D+mH)$ on $Y$ of rank $3$, and let $Z=\mathbb{P}_{Y}(\mathcal{E})\stackrel{p}{\longrightarrow} Y$ be the projection. Set $\xi:=\mathcal{O}_{Z}(1)$.  Then the general divisor $X\in |2 \xi-2mp^{*}H|$ is smooth, and $p_{\mid X}\colon X\to Y$ is an elementary Fano conic bundle.  
\end{proposition}
\begin{proof} 
	By Lemma \ref{base_locus} we deduce that $\mathcal{E}\otimes \mathcal{O}_{Y}(-mH)=\mathcal{O}_{Y}(D-mH)^{\oplus 2}\oplus \mathcal{O}_{Y}(D)$ is globally generated outside $V$, and by \cite[Lemma 2.3.2]{LAZ} the tautological bundle $\xi-mp^{*}H$ of $\mathbb{P}_{Y}(\mathcal{E}\otimes \mathcal{O}(-mH))$ is spanned outside $p^{-1}(V)$. 
	
	Being $\Bs{(|2 \xi-2mp^{*}H|)}\subseteq \Bs{(|\xi-mp^{*}H|)}$, also $|2 \xi-2mp^{*}H|$ is globally generated outside $p^{-1}(V)$; in particular $|2 \xi-2mp^{*}H|\neq\emptyset$. 
	
	By Bertini's theorem it follows that a general divisor $X\in |2\xi-2mp^{*}H|$ is smooth outside $p^{-1}(V)$. 
	
	Now we prove that $-K_{Z}-X$ is ample on $Z$. This will imply that $X$ is Fano, once that we show the smoothness of $X$. We have: 
	\begin{center}
		$-K_{Z}=3 \xi-p^{*}(K_{Y}+3D+mH)=3 \xi+(1-2m)p^{*}H$.
	\end{center}
	
	Being $X\in |2 \xi-2mp^{*}H|$, one has $-K_{Z}-X=\xi+p^{*}H$. 
	
	Notice that $\xi+p^{*}H$ is the tautological bundle of $\mathbb{P}_{Y}(\mathcal{E}\otimes \mathcal{O}(H))=\mathbb{P}_{Y}(\mathcal{O}(D+H)^{\oplus 2}\oplus \mathcal{O}(D+(m+1)H))$ and being $\mathcal{E}\otimes \mathcal{O}(H)$ sum of ample line bundles of $Y$, we conclude that $-K_{Z}-X$ is ample on $Z$. 
	
	Now we show that the fibers of $p\colon X\to Y$ are isomorphic to plane conics.
	
	We observe that the divisor $X$ is given by a section of $H^{0}(Z, \mathcal{O}_{Z}(2 \xi-2mp^{*}H))\cong H^{0}(Y,\Sym^{2}{(\mathcal{E}\otimes \mathcal{O}(-mH)))}$. Being 
	\begin{center}
	$\Sym^{2}{(\mathcal{E}\otimes \mathcal{O}(-mH)))}= \mathcal{O}_{Y}(2(D-mH))^{\oplus 3}\oplus \mathcal{O}_{Y}(2D-mH)^{\oplus 2}\oplus \mathcal{O}_{Y}(2D)$, 
	\end{center}
	$X$ can be represented by a symmetric matrix of sections:
	$$
	S=\left( \begin{array}{ccc}
	s_{1} & s_{2} & \lambda_{1} \\
	s_{2} & s_{3} & \lambda_{2} \\
	\lambda_{1} & \lambda_{2} & \sigma
	\end{array}\right) 
	$$
	with $s_{1}, s_{2}, s_{3} \in H^{0}(Y, \mathcal{O}_{Y}(2(D-mH)))$, $\lambda_{1}, \lambda_{2}\in H^{0}(Y,\mathcal{O}_{Y}(2D-mH))$ and $\sigma\in H^{0}(Y,\mathcal{O}_{Y}(2D))$.
	We prove that the matrix $S$ vanishes nowhere. 
	
	Since $D$ is globally generated and $D_{\mid V}\cong \mathcal{O}_{V}$, $\sigma$ can be chosen such that $\sigma_{\mid V}$ is constant and non-zero.
	
	Then the possible common zeros of the sections of $S$ are outside $V$. But by Lemma \ref{base_locus} we know that $\Bs{(|2D-2mH|)}=\Bs{(|2D-mH|)}=V$, hence for a general choice of these sections they have no common zeros and this allows to deduce that a general divisor $X\in |2 \xi-2mp^{*}H|$ gives a conic fibration over $Y$. 
	
	It remains to show that $X$ is smooth in $p^{-1}(V)\cap X$. To this end, we take an open subset $U\subset Y$ which trivializes $\mathcal{E}$. If $[z_{0}:z_{1}:z_{2}]$ are projective coordinates of the trivialization of $\mathbb{P}(\mathcal{E})$ over $U$ compatible with the splitting of $\mathcal{E}$, then on $p^{-1}(U)\cong U\times \mathbb{P}^{2}$, the general divisor $X\in \ |2\xi-2mp^{*}H|$ is defined by the equation:
	\begin{center} 
	$F:= (z_{0}, z_{1}, z_{2}) \ S \ (z_{0}, z_{1}, z_{2})^{t}=0$,
	\end{center}
	namely:
	\begin{center}
		$F= s_{1} z_{0}^{2}+ 2 s_{2} z_{0} z_{1}+ 2 \lambda_{1} z_{0} z_{2}+ s_{3} z_{1}^{2}+ 2 \lambda_{2} z_{1} z_{2}+ \sigma z_{2}^{2}=0$. 
	\end{center}
	
	Using Lemma \ref{base_locus}, the equation $F$ over $V$ degenerates to $\sigma z_{2}^{2}=0$. Therefore, to prove the smoothness of a general $X$ we have to compute the differential $dF$ over $W:= p^{-1}(V)\cap X= p^{-1}(V)\cap \{z_{2}=0\}$. 
	
	
	
	Taking the covering of $\mathbb{P}^{2}$ given by the standard open subsets we can compute $dF$ locally over each $U\times B_{i}$ where $B_{i}=\{(z_{0}:z_{1}:z_{2})\in \mathbb{P}^{2}\mid z_{i}\neq 0\}\cong \mathbb{A}^{2}$ for every $i=0,1,2$. 
	
	Take $U\times B_{0}$, where $B_{0}\cong \mathbb{A}^{2}_{y,z}$, with coordinates $y=\frac{z_{1}}{z_{0}}$, $z=\frac{z_{2}}{z_{0}}$. We evaluate the differential of this equation over $(U\times B_{0}) \cap \{z=0\}$. Considering all the vanishing of the sections of the matrix $S$ (that are regular functions over $U$) due to Lemma \ref{base_locus}, this differential becomes:
	\begin{center}
	$ds_{1}+ 2 y ds_{2}+ y^{2} ds_{3}$.
	\end{center}
	In the same way, we consider $U\times B_{1}$, where $B_{1}\cong \mathbb{A}^{2}_{x,z}$, with coordinates $x=\frac{z_{0}}{z_{1}}$, $z=\frac{z_{2}}{z_{1}}$. Then we evaluate $dF$ over $(U\times B_{1}) \cap \{z=0\}$, getting:
	\begin{center}
	$x^{2} ds_{1}+ 2 x ds_{2}+ ds_{3}$. 
	\end{center}
	By these local computations, it follows that:
	\begin{center}
		$dF_{\mid W}= z_{0}^{2} ds_{1}+2 z_{0} z_{1} ds_{2}+ z_{1}^{2} ds_{3}$.
	\end{center}
	
	We have already observed that $V\subset Y$ is given by a complete intersection of two divisors $G_{1}, G_{2}\in |D- 2mH|$. Let $\nu_{1}$, $\nu_{2}$ be the two sections of $H^{0}(Y, \mathcal{O}_{Y}(D-2mH))$ corresponding to these two divisors.
	
	Since $G_{1}$ and $G_{2}$ intersect transversally along $V$, their differentials $d\nu_{1}$, $d\nu_{2}$ are independent everywhere on $V$. 
	
	Now, let $\sigma^{\prime} \in H^{0}(Y, \mathcal{O}_{Y}(D))$ be a section such that $\sigma^{\prime}_{\mid V}$ is constant and non-zero, and set $s_{1}=\sigma^{\prime} \nu_{1}$ and  $s_{2}=s_{3}=\sigma^{\prime} \nu_{2}$. 
	
	Then ${ds_{1}}_{\mid V}=\sigma^{\prime}_{\mid V} {d\nu_{1}}_{\mid V}$, ${ds_{2}}_{\mid V}= {ds_{3}}_{\mid V}={\sigma^{\prime}}_{\mid V} {d\nu_{2}}_{\mid V}$, and:
	\begin{center}
	$dF_{\mid W}={\sigma^{\prime}}_{\mid V} (z_{0}^{2}d\nu_{1}+z_{1} (2z_{0}+z_{1})d\nu_{2})$ 
	\end{center}
	does not vanish on $W$. Thus $X$ is smooth, so that $p_{\mid X}$ is a Fano conic bundle. By Theorem \ref{esempio} we conclude that $\rho_{X}-\rho_{Y}=1$ because $\rho_{Y}=2$, hence the statement.
\end{proof}

Using the above construction and results we prove Theorem \ref{new_examples}. 
\begin{proof} [Proof of Theorem \ref{new_examples}] Setting as in  Proposition \ref{prop_ex}. Take $f:=p_{\mid X}$. 
	
We have already proved that $f\colon X\to Y$ is an elementary Fano conic bundle and that $-K_{Y}$ is not nef. Finally, since $\pi\colon Y\to \mathbb{P}^{3m}$ is a $\mathbb{P}^{2}$-bundle, $\dim{(Y)}=3m+2$, and $\dim{(X)}=3(m+1)$.
\end{proof}
\begin{remark}
	Keeping the notation introduced in this subsection, by the proof of Proposition \ref{prop_ex} we can observe that the fibers of $f\colon X\to Y$ over $V$ are not reduced, so that by (\ref{singular_discriminant}), it follows that $V\subseteq \Sing{(\bigtriangleup_{f})}$. 
	
	Moreover, by (\ref{discriminant}) we know that $\bigtriangleup_{f}\sim 2 \det{\mathcal{E}}+3M$ where in our case $\mathcal{E}:=\mathcal{O}(D)^{\oplus 2}\oplus \mathcal{O}(D+mH)$, $M\sim -2mH$, so that $\bigtriangleup_{f}\sim 2(3D-2mH)$, and being $\Bs{(|2(3D-2mH)|)}\subseteq \Bs{(|3D-2mH|)}\subseteq \Bs{(|D-2mH|)}=V$, we get $\Sing{(\bigtriangleup_{f})}=V$. Finally, if $C\subset V\cong \mathbb{P}^{3m}$ is a line, we can compute that $\bigtriangleup_{f}\cdot C=-4m$.
\end{remark}
\section{Final Comments and Open Questions} \label{conclusion}

\noindent Let us denote by $f\colon X\to Y$ a non-elementary Fano conic bundle, and by $\bigtriangleup_{f}$ its discriminant divisor. Thanks to Theorem \ref{esempio} the positivity properties of $-K_{Y}$ are well known when $\rho_{X}-\rho_{Y}\geq 3$. The more mysterious case is when $\rho_{X}-\rho_{Y}= 2$. 

In particular, we do not know if there exists a Fano conic bundle $f\colon X\to Y$ with non-reduced fibers and $\rho_{X}-\rho_{Y}=2$. Using the classification of Fano 3-folds of \cite{MM2, MM}  and \cite[Theorem 1.1]{IO} it is easy to check that there are no examples of Fano conic bundles $f\colon X\to S$ where $\dim{(X)}=3$, $\rho_{X}-\rho_{Y}=2$ and $f$ has non-reduced fibers. Our questions are the following:

\begin{question} Are there examples of Fano conic bundles $f\colon X\to Y$ where $\rho_{X}-\rho_{Y}=2$ and $f$ has non-reduced fibers? 
\end{question}

\begin{question}
	Are there examples of Fano conic bundles $f\colon X\to Y$ where $\rho_{X}-\rho_{Y}=2$ and $Y$ is not Fano?
\end{question}
Notice that the two questions are related by Corollary \ref{Fano}. 

The following corollary is a consequence of some results of Viehweg and Koll$\acute{\text{a}}$r (see \cite{ANDREAS} and references therein). In particular, it follows applying \cite[Theorem 3.30]{ANDREAS} and \cite[Lemma 3.5]{ANDREAS} when $f\colon X\to Y$ is a Fano conic bundle:
\begin{corollary} \label{Andreas}
	Let $f\colon X\to Y$ be a Fano conic bundle. Denote by $\bigtriangleup_{f}$ its discriminant divisor. If $C\subset Y$ is an irreducible curve such that $-K_{Y}\cdot C<0$, then $C\subset \bigtriangleup_{f}$.
\end{corollary}

On the other hand, by Proposition \ref{PropWis} we know that \textit{rational} curves $C\subset Y$ with $-K_{Y}\cdot C<0$ cover at most the closed subset $\Sing{(\bigtriangleup_{f})}$ which is properly contained in $\bigtriangleup_{f}$. 

By Corollary \ref{Fano} it seems that there is a close link between the kind of singularities of the fibers of $f$ and the positivity properties of $-K_{Y}$. This remark induces the following question posed by Ejiri in \cite{EJIRI}, where we refer the reader to \cite[Definition 1.3]{ALEXEEV} for the definition of semi-log canonical singularities. 
\begin{question} \label{positivity_question}
	Let $X$ be a smooth Fano variety. Let $f\colon X\to Y$ be an equidimensional fiber type contraction. Assume that $f$ is not a smooth morphism, but its fibers have some mild singularities, for example semi-log canonical singularities. Does $-K_{Y}$ have some good positivity properties?
\end{question} 

As observed in \cite[Remark 2.5]{IO}, when $f\colon X\to Y$ is a Fano conic bundle, Corollary \ref{Fano} gives a positive answer to Ejiri's question. Indeed, the fibers of the conic bundle $f$ have semi-log canonical singularities if and only if they are reduced (see \cite[Example 1.4]{ALEXEEV}).
\begin{remark}
	As we observed in the Introduction, by \cite[Theorem 2.9]{PROK} we know that the target of a contraction of a log Fano variety is log Fano, so that in the setting of Question \ref{positivity_question} (and thus in the specific case in which $f$ is a Fano conic bundle) one has that $-K_{Y}$ is always a big divisor of $Y$.
\end{remark}
\begin{remark} Notice that if in Question \ref{positivity_question} we do not assume that $f$ is equidimensional, it is easy to find counterexamples. For instance in \cite[Example 3.16 (2)]{DEB} we find examples of divisorial contractions $\varphi\colon Z\to Y$, where $Z$ is smooth and Fano, the fibers of $\varphi$ are smooth, but $Y$ is not Fano. If we take the natural projection $\pi\colon\mathbb{P}^{1}\times Z\to Z$, then $\varphi\circ \pi\colon \mathbb{P}^{1}\times Z\to Y$ gives a fiber type contraction of a Fano variety, where the fibers are smooth, but $Y$ is not Fano. 
\end{remark}

\textit{Acknowledgements}. This paper is part of my Ph.D. thesis. I would like to thank my advisor Cinzia Casagrande for her constant guidance, support and for many valuable conversations and advices.

\nocite{*}
\bibliographystyle{plain} 
\bibliography{biblio}

\end{document}